\documentclass[12pt,a4paper]{amsart}

\usepackage{amssymb}
\usepackage[final]{showkeys}
\usepackage{microtype}
\usepackage{color}
\usepackage{graphicx}
\usepackage[hmargin=2cm,vmargin={3.5cm,4cm}]{geometry}

\newtheorem{te}{Theorem}[section]

\newtheorem{definition}[te]{Definition}
\newtheorem{os}[te]{Remark}
\newtheorem{prop}[te]{Proposition}

\newtheorem{coro}[te]{Corollary}

\numberwithin{equation}{section}

\allowdisplaybreaks

\begin{document}

    \title[Some probabilistic properties of fractional point processes]{Some probabilistic properties of fractional point processes}

	\author{Roberto Garra$^1$}
        \address{${}^1$Dipartimento di Scienze Statistiche, ``Sapienza'' Universit\`a di Roma.}

    \author{Enzo Orsingher$^1$}
    
    \author{Marco Scavino $^{2,3}$}
    	    \address{${}^2$ CEMSE, King Abdullah University of Science and Technology, Thuwal, Saudi Arabia
    		 \newline ${}^3$ Instituto de Estad\'istica (IESTA), Universidad de la Republica, Montevideo, Uruguay
    	    }


    \date{\today}

    \begin{abstract}
	This paper studies the first hitting times of generalized Poisson processes $N^f(t)$, related to Bern\u{s}tein functions $f$. For the
	space-fractional Poisson processes, $N^\alpha(t)$, $t>0$ (corresponding to $f= x^\alpha$), the hitting probabilities $P\{T_k^\alpha<\infty\}$
	are explicitly obtained and analyzed. The processes $N^f(t)$ are 
	time-changed Poisson processes $N(H^f(t))$ with subordinators $H^f(t)$
	and here we study $N\left(\sum_{j=1}^n H^{f_j}(t)\right)$ and obtain probabilistic features of these extended counting processes.
	A section of the paper is devoted to processes of the form $N(\mathcal{G}_{H,\nu}(t))$ where $\mathcal{G}_{H,\nu}(t)$ are generalized grey Brownian motions. This involves the theory of time-dependent fractional operators of the McBride form. 
	While the time-fractional Poisson process is a renewal process, we prove that the space-time Poisson
	process is no longer a renewal process. 
	
	\smallskip

    \end{abstract}

    \maketitle

    \section{Introduction}
    Fractional Poisson processes of various forms have been introduced and studied in the last years by many researchers (for example, the 
    time-fractional Poisson process in \cite{Laskin,Scalas,be}, the space-fractional Poisson process in \cite{macci,fede,Enrico}).
     
    The space-fractional Poisson process, introduced in \cite{fede} is a time-changed Poisson process $N(H^{\alpha}(t))= N^{\alpha}(t)$
    where $H^{\alpha}(t)$ is a stable subordinator, independent from the homogeneous Poisson process $N(t)$. 
    In a paper appeared at the end of 2015 \cite{bruno}, a new class of time-changed Poisson processes $N(H^{f}(t))= N^f(t)$,
    including the space-fractional Poisson process as a special case, has been introduced.
    Here $H^{f}(t)$, $t>0$ is a subordinator related to the Bern\u{s}tein function $f$ with Laplace transform 
    \begin{equation}
    \mathbb{E} e^{-\mu H^f(t)}= e^{-t f(\mu)} =e^{-t\int_0^{\infty}(1-e^{-\mu s})\nu(ds)},
    \end{equation}
    where $\nu(ds)$ is the related L\'evy measure. The point processes 
    $N^f(t)$, $t>0$, have jumps of arbitrary size whose distribution 
    is given in formula \eqref{for1} below.\\
    In the first part of the paper we study the first-passage times
    \begin{equation}
    T_k^{\alpha}:= \inf \left(s:N^{\alpha}(s)= k\right), \quad k\geq 1
    \end{equation}
    for the space-fractional Poisson process. We are able to give the explicit form of the hitting probabilities 
    \begin{equation}
    P\{T^{\alpha}_k<\infty\}= \frac{\Gamma(k+\alpha)}{\Gamma(\alpha)}\frac{1}{k!}<1, \quad \mbox{for all $k\geq 1$}.
    \end{equation}
    The hitting probabilities $P(T_k^{\alpha}<\infty)$ are strictly less than one for all $k\geq 1$ because $N^{\alpha}(t)$ is a process with 
    independent increments and jumps of arbitrary size and thus - with positive probability - can skip over every level $k$.
    For the process $N^f(t)$, $t>0$, we will show that for the first-passage times $T_k^f$ it is not possible to obtain the explicit 
    value of $P\{T_k^f<\infty\}$, for arbitrary Bern\u{s}tein functions
    $f$. However we are able to show that $P\{T_1^f<\infty\}<1$ for all $f$.\\
    For the space-fractional Poisson process $N^{\alpha}(t)$ we study the $n$-times iterated process
    \begin{equation}\label{1}
    N^{\alpha}(H^{\alpha_1}(H^{\alpha_2}(\dots H^{\alpha_n}(t))\dots)),
    \end{equation}
    with $H^{\alpha_j}(t)$, independent stable subordinators and analyse its limiting process.
    Under suitable conditions, the process defined in \eqref{1} for $n\rightarrow +\infty$ converges to a space-fractional 
    Poisson process. This property does not hold for the iterated process
    \begin{equation}
        N^{\alpha}(H^{f_1}(H^{f_2}(\dots H^{f_n}(t))\dots)),
        \end{equation}
    where $f_j$ are arbitrary Bern\u{s}tein functions and $f_j(x)\neq x^{\alpha_j}$.\\
    It is well-known that the time-fractional Poisson process (see e.g. \cite{be},\cite{gorenflo}) is a renewal process 
    with Mittag-Leffler distributed intertimes, that is the random times $U_j$ between the $j$-th and $(j+1)$-th
    event have distribution
    \begin{equation}
    P\{U_j>t\}= E_{\nu,1}(-\lambda t^{\nu}), \quad \lambda>0, \nu \in (0,1),
    \end{equation}
    for all $j\geq 0$, where $E_{\nu,1}(x)= E_{\nu}(x)$ is the classical Mittag-Leffler function. \\
    We will show, instead, that the space-time fractional Poisson process $N^{\alpha, \nu}(t)$ with probability generating
    function (p.g.f.)
    \begin{equation}
    G_{\alpha, \nu}(t)= \mathbb{E}u^{N^{\alpha,\nu}(t)}= E_{\nu, 1}(-\lambda^{\alpha}(1-u)^\alpha t^\nu),
    \end{equation}
    does not have the structure of a renewal process.
    We also note that the process $N^{f,\nu}(t)$ with probability generating function 
    \begin{equation}
        G_{f, \nu}(t)=  E_{\nu, 1}(-f(\lambda(1-u)) t^\nu),
    \end{equation}
    for all Bern\u{s}tein functions $f$, does not possess the structure of renewal processes.\\
    Generalizations of the time-changed Poisson process $N^f(t)$ have been considered
    in the form $N(\sum_{j=1}^n H^{f_j}(t))$, where $H^{f_j}(t)$ are independent subordinators related to the
    Bern\u{s}tein functions $f_j$ and independent from the homogeneous Poisson process $N(t)$ (see also \cite{Enrico}).\\
    In some papers (see e.g. \cite{stoca}) time-changed Poisson processes
    where the role of time is played by some form of Brownian motions
    (for example, the elastic Brownian motion) or functionals of Brownian 
    motions have been analyzed. 
    In \cite{be} it was proved that the time-fractional Poisson process $N^{\nu}(t)$, $t>0$ for $\nu = 1/2$
    coincides in distribution with 
    \begin{equation}
    N^{1/2}(t)\stackrel{d}{=}N\left(|B(t)|\right).
    \end{equation}
    Inspired by this result we consider here the Poisson process where the role of time is played 
    by the generalized grey Brownian motion. The generalized grey Brownian motion, $\mathcal{G}_{H,\nu}(t)$,
    has probability distribution $u(y,t)$ satisfying the following fractional equation (see \cite{jmp})
    \begin{equation}\label{2}
    	{}^C\left(t^{1-2H}\frac{\partial}{\partial t}\right)^{\nu} u(y,t)=
    	c^2\frac{\partial^2}{\partial y^2}u(y,t), \quad \nu \in (0,1), H\in (0,1), y\in\mathbb{R}, t>0.
    	\end{equation}
	The fractional operator appearing in \eqref{2} denotes the $\nu$-th order regularized Caputo-like operator, see the Appendix for further 
	details on this point.   
	We show that the probabilities 
	\begin{equation}
	p_k(t)= P\{N\left(|\mathcal{G}_{H,\nu}(t)|\right)= k\}
	\end{equation}    
	satisfy the difference-differential equations
	\begin{equation}
	{}^C\left(t^{1-2H}\frac{\partial}{\partial t}\right)^{\nu/2}p_k(t)=	-\lambda' p_k(t)+\lambda' p_{k-1}, \quad k\geq 0,
	\end{equation}
	with $\lambda' = \lambda (2H)^{\nu/2}$.
    The treatment of the time-fractional operator appearing in \eqref{2} can be performed by applying the theory 
    developed by McBride and coauthors (see e.g. \cite{Mcbride},\cite{jmp} and the Appendix for details on this point).
    For $\nu = 2$, we have, as special case, the Poisson process at a reflected fractional Brownian motion, i.e. $N(|B_H(t)|)$.

	\section{First-passage times of the space-fractional Poisson process}
	
	The space-fractional Poisson process $N^{\alpha}(t)$, $t>0$, 
	first introduced by Orsingher and Polito in \cite{fede}, has state probabilities $p_k^{\alpha}(t)$ satisfying the following infinite
	system of difference-differential equations
	\begin{equation}\label{s-f}
	\begin{cases}
	\frac{d}{dt}p_k^{\alpha}(t)=-\lambda^{\alpha}(1-B)^{\alpha}p_k^{\alpha}(t), \quad \alpha \in (0,1]\\
	p_k^{\alpha}(0)=\begin{cases}
	0 & k>0\\
	1 & k=0
	\end{cases}
	\end{cases}
	\end{equation}
	 where $B$ is the classical backward-shift operator ($Bp_k(t)= p_{k-1}(t)$).
	 The process $N^{\alpha}(t)$ has independent increments; probability distribution
	\begin{align}
 		\label{stato} &P\{N^{\alpha}(t)=k\}= p_k^{\alpha}(t)= \frac{(-1)^k}{k!}\sum_{r=0}^{\infty}
 		\frac{(-\lambda^{\alpha}t)^r}{r!}\frac{\Gamma(\alpha r+1)}{\Gamma(\alpha r+1-k)}\\
 		\nonumber &=\frac{(-1)^k}{k!}\frac{d^{k}}{du^{k}}
 				e^{-t\lambda^{\alpha}u^{\alpha}}\bigg|_{u=1}, \qquad  k \geq 0
	\end{align} 
	and probability generating function
	\begin{equation}
	\mathbb{E} u^{N^{\alpha}(t)}= e^{-t\lambda^{\alpha}(1-u)^{\alpha}}, \quad |u|\leq 1.
	\end{equation}	
	Furthermore
	\begin{equation}\label{salti}
	 P\{N^{\alpha}[t,t+dt)=k\}=
	 \begin{cases} \frac{(-1)^{k+1}\lambda^{\alpha}\Gamma(\alpha+1)}{k!\Gamma(\alpha+1-k)} \, dt \qquad  k \geq 1\\
	1-\lambda^{\alpha}\,dt\qquad k=0.
	\end{cases}
	\end{equation}
    and
\begin{equation}
\mathbb{E} \ N^{\alpha}(t)= \mbox{Var} \ N^{\alpha}(t)= +\infty, \quad \forall \ \alpha \in (0,1), \ t>0.
\end{equation}	

   In the next Theorem we present the hitting probabilities 
   $P\{T_k^{\alpha}<\infty\}$
   of the first passage times of the space-fractional Poisson process $N^{\alpha}(t)$ defined as
   \begin{equation}
   T_k^{\alpha}:= \mbox{inf} \ (s:N^{\alpha}(s)=k), \quad k\geq 1.
   \end{equation}
    
	\begin{te}
	For the random times $T_k^{\alpha}$ we have that 
	\begin{equation}\label{bel}
	P\{T_k^{\alpha} <\infty\} = \frac{\alpha(\alpha+1)\dots(\alpha+k-1)}{k!}=\frac{\Gamma(k+\alpha)}{\Gamma(\alpha)}\frac{1}{k!}<1, \quad \forall \ k\geq 1,	
	\end{equation}
	when $\alpha\in(0,1)$. Clearly $P\{T_k^{1} <\infty\} = 1$.
	\end{te} 
	
	\begin{proof}
	The hitting time distribution of the space-fractional Poisson process
	is given by
	\begin{equation}
	P\{T_k^{\alpha}\in ds\}= P\bigg\{\bigcup_{j=1}^k \{N^{\alpha}(s)=k-j, N^{\alpha}[s,s+ds)=j\}\bigg\}
	\end{equation}
	and by the independence of the increments and by using \eqref{salti}
	and \eqref{stato}, we have
	\begin{align}
	P\{T_k^{\alpha}\in ds\}&= \sum_{j=1}^{k}P\{N^{\alpha}(s)= 
	k-j\}P\{N^{\alpha}[s,s+ds)=j\}\\
	\nonumber	&=\sum_{j=1}^k\frac{(-1)^{k-j}}{(k-j)!}\frac{d^{k-j}}{du^{k-j}}
		e^{-s\lambda^{\alpha}u^{\alpha}}\bigg|_{u=1} \frac{\Gamma(\alpha+1)}{j!\Gamma(\alpha+1-j)}(-1)^{j+1}\lambda^{\alpha}ds.
	\end{align} 
	Therefore
	\begin{align}
	\label{a}&P\{T_k^{\alpha}<\infty\}= \sum_{j=1}^k\frac{(-1)^{k+1}}{(k-j)!}\frac{\Gamma(\alpha+1)}{j!\Gamma(\alpha+1-j)}\frac{d^{k-j}}{du^{k-j}}u^{-\alpha}\bigg|_{u=1}\\
	\nonumber &= (-1)^{k+1}\frac{\alpha}{k!}\sum_{j=1}^{k}\binom{k}{j}
	\frac{\Gamma(\alpha+k-j)}{\Gamma(\alpha+1-j)}(-1)^{k-j}.
	\end{align}
	By using the identity
	\begin{equation}
	\Gamma(\alpha+1-j)= \frac{\pi}{\sin \pi(j-\alpha)}\frac{1}{\Gamma(j-\alpha)}
	\end{equation}
	and the fact that 
	\begin{equation}
	\sin\pi(j-\alpha)= (-1)^{j+1}\sin \pi \alpha,
	\end{equation}
	we obtain that 
	\begin{equation}\label{b}
	P\{T_k^{\alpha}<\infty\}=\frac{\alpha\Gamma(k)}{k!}\frac{\sin \pi \alpha}{\pi}\sum_{j=1}^{k}\binom{k}{j}
		B(\alpha+k-j, j-\alpha),
	\end{equation}
	where we have multiplied and divided for $\Gamma(k)$ inside the summation formula appearing in \eqref{a}. 
	We have denoted with $B(\alpha+k-j, j-\alpha)$ the Beta function, whose integral representation leads to the following equivalent form of
	formula \eqref{b}
	\begin{align}
	&P\{T_k^{\alpha}<\infty\}=\frac{\alpha\Gamma(k)}{k!}\frac{\sin \pi \alpha}{\pi}\sum_{j=1}^{k}\binom{k}{j}\int_0^1 x^{\alpha+k-j-1}
	(1-x)^{j-\alpha-1}dx\\
    \nonumber &= \frac{\alpha\Gamma(k)}{k!}\frac{\sin \pi \alpha}{\pi}\int_0^1 x^{\alpha-1}
    	(1-x)^{-\alpha-1}\sum_{j=1}^{k}\binom{k}{j}x^{k-j}(1-x)^jdx\\
	\nonumber &= -\frac{\alpha\Gamma(k)}{k!}\frac{\sin \pi \alpha}{\pi}\frac{\Gamma(k+\alpha)\Gamma(-\alpha)}{\Gamma(k)}.	
	\end{align}
	 Finally using the following equality
	 \begin{equation}
	 -\alpha\Gamma(-\alpha)= \Gamma(1-\alpha)= \frac{\pi}{\sin \pi \alpha}
	 \frac{1}{\Gamma(\alpha)},
	  \end{equation}
	  we obtain the claimed result.
	  \end{proof}
	\begin{os}
	Since 
	\begin{equation}
	\sum_{k=1}^{\infty}\frac{\Gamma(k+\alpha)}{\Gamma(k+1)}\geq \sum_{k=1}^{\infty}\frac{\Gamma(k)}{\Gamma(k+1)}= \sum_{k=1}^{\infty}\frac{1}{k}= \infty,
	\end{equation}
	we have that 
	\begin{equation}
	\lim_{k\rightarrow 0}\frac{\Gamma(k+\alpha)}{\Gamma(\alpha)k!}\rightarrow 0,
	\end{equation}
	slowly with $k$.
	\end{os}

	\begin{coro}
	We have that 
	\begin{equation}
	P\{T_k^{\alpha}<\infty\}<1, \quad \forall \ k\geq 1 
	\end{equation}
	and 
	\begin{equation}
	P\{T_k^{\alpha}<\infty\}< 	P\{T_{k-1}^{\alpha}<\infty\}, \forall \ k\geq 1
	\end{equation} 
		Observe also that, from \eqref{bel} the following interesting recursive
		relationship holds
			\begin{equation}
			P\{T_k^{\alpha}<\infty\}= \left[\frac{\alpha+k-1}{k}\right]P\{T_{k-1}^{\alpha}<\infty\}
			\end{equation}
	\end{coro}

	\begin{os}
	We note that the hitting probabilities \eqref{bel} are increasing 
	functions of $\alpha$ since
	\begin{equation}\label{psi}
	\frac{d}{d\alpha}\frac{\Gamma(k+\alpha)}{\Gamma(\alpha)}= 
	\frac{\Gamma(k+\alpha)}{\Gamma(\alpha)}\bigg[\psi(k+\alpha)-\psi(\alpha)\bigg]>0,
	\end{equation}
	where the function 
	\begin{equation}\nonumber
	\psi(x)= \frac{d}{dx}\ln \Gamma(x),
	\end{equation}
	has different representations. Furthermore, in order to prove the inequality \eqref{psi}, we recall the following useful representation 
	(see \cite{grad}, pag.944, formula 8.363, n.3)
	\begin{equation}\nonumber
	\psi(k+\alpha)-\psi(\alpha)= \sum_{r=0}^{\infty}\left(\frac{1}{\alpha+r}-\frac{1}{\alpha+r+k}\right)
	\end{equation}
	\end{os}
	
	On the basis of the result of Theorem 2.1, we can give some asymptotic estimates of the hitting probabilities $P\{T_k^\alpha <\infty \}$ for large values of $k$.
	
	\begin{prop}
	For $k\rightarrow +\infty$, we have the following asymptotic estimate
	\begin{equation}\label{stima}
	P\{T_k^{\alpha}<\infty\}\sim \frac{e^{-\alpha}}{\Gamma(\alpha)}\frac{1}{k^{1-\alpha}}, \quad \alpha\in (0,1).
	\end{equation}
	\end{prop}
	
	\begin{proof} 
	We will use the following asymptotic approximation for the Gamma function 
	\begin{equation}
	\Gamma(x)	\sim x^{x-\frac{1}{2}}e^{-x}\sqrt{2\pi}, \quad \mbox{for $x\rightarrow +\infty$,}
	\end{equation}   
	and the classical Stirling formula, in order to obtain result \eqref{stima}
		\begin{align}
		P\{T_k^{\alpha} < \infty\} &= \frac{\Gamma(k+\alpha)}{\Gamma(\alpha)}\frac{1}{k!}
		\sim \frac{(k+\alpha)^{k+\alpha-\frac{1}{2}}e^{-(k+\alpha)}}{\Gamma(\alpha)k^{k+1/2}e^{-k}} \\
		\nonumber &\sim \frac{e^{-\alpha}}{\Gamma(\alpha)}\frac{1}{k^{1-\alpha}}.
		\end{align}
	\end{proof}
	
	Observe that, for $\alpha = 1/2$, the hitting probability has the following representation
	\begin{equation}\label{mezzo}
			P\{T_k^{1/2}<\infty\}= \binom{2k}{k}\frac{1}{2^{2k}}, \quad k >0 
	\end{equation}
	and	
	\begin{equation}
	P\{T_k^{1/2}<\infty\}\sim \frac{e^{-1/2}}{\sqrt{\pi k}}, \quad \mbox{for $k\rightarrow +\infty$}.
	\end{equation}
	Equation \eqref{mezzo} coincides with
	$P\{\mathcal{B}(2k,\frac{1}{2})=0\}$, where $\mathcal{B}(\cdot,\cdot)$ stands for the binomial distribution.
	
	\subsection{Iterated compositions}
	
	As previously noticed, the space-fractional Poisson process $N^{\alpha}(t)$ can be regarded as an homogeneous Poisson process $N(t)$ subordinated to a positively-skewed stable subordinator $H^{\alpha}(t)$, with $\alpha \in (0,1)$. In \cite{fede}, it was also proved (Theorem 2.3)
	that, given a $\gamma$-stable subordinator and a space-fractional Poisson process $N^{\alpha}(t)$, the following equality in distribution holds
	\begin{equation}\label{fedez}
	N^{\alpha}(H^{\gamma}(t))\stackrel{d}{=}N^{\alpha \gamma}(t).
	\end{equation}
	Our aim is to generalize this observation, by considering the composition of a space-fractional Poisson process with $n$ independent stable subordinators
	$H^{\gamma_j}(t)$. From equation \eqref{fedez} we have that, for any value of $n$, the following equality in distribution holds 
	\begin{equation}
	N^{\alpha}(H^{\gamma_1}(H^{\gamma_2}(\dots H^{\gamma_n}(t))))
	\stackrel{d}{=} N^{\alpha \prod_{j=1}^n \gamma_j}(t),
	\end{equation}
	that is the $n$-fold subordination of a space-fractional Poisson process $N^\alpha$
	yields a space-fractional Poisson process of degree $\alpha = \displaystyle{\prod_{j=1}^n} \gamma_j$.\\
	For $n\rightarrow +\infty$, the \textit{infinitely stable subordinated} space-fractional Poisson process is a space-fractional Poisson process
	if $1>\displaystyle{\prod_{j=1}^{\infty}}\gamma_j>0$. This happens, for example for
	$\gamma_j = e^{-\mu_j}$ with $\displaystyle{\sum_{j=1}^{\infty}}\mu_j<\infty$. 
	If $\displaystyle{\prod_{j=1}^{\infty}}\gamma_j= 0 $ we have a degenerate case with a r.v. taking value $0$ with probability $e^{-t}$
	and $\infty$ with probability $1-e^{-t}$.
	This auto-conservative property of the space-fractional Poisson process does not hold for the other counting processes with Bern\u{s}tein intertimes considered in Section 3.

	\subsection{Relationship between space-time fractional Poisson process and renewal process with Mittag-Leffler intertimes}

	The space-time fractional Poisson process is simply obtained by replacing in equation \eqref{s-f} the ordinary time derivative
	with the Caputo time-fractional derivative of order $\nu \in (0,1]$. 
	The space-time fractional Poisson process is a time-changed Poisson process of the form $N^{\alpha,\nu}(t)= N(H^{\alpha}(L^\nu(t)))$
	where $L^\nu(t)$, $0<\nu<1$, is the inverse of the stable subordinator
	$H^\nu(t)$ and is independent from $H^\alpha$ and $N$.
	The distribution $p_k^{\alpha,\nu}(t)$ of $N^{\alpha,\nu}(t)$ reads
	\begin{align}
		p_k^{\alpha,\nu}(t)= P\{N^{\alpha,\nu}(t)=k\}= \frac{(-1)^k}{k!}\sum_{m=0}^{\infty}
				\frac{(-\lambda^{\alpha}t^{\nu})^m}{\Gamma(\nu m+1)}
				\frac{\Gamma(\alpha m+1)}{\Gamma(\alpha m+1-k)}
	\end{align}
	and the related probability generating function has the form 
	\begin{equation}
	G(u,t)= E_{\nu,1}(-\lambda^{\alpha}(1-u)^{\alpha}t^{\nu}), \quad |u|\leq 1.
	\end{equation}
	Let us construct a renewal process $\mathcal{N}^{\alpha,\nu}(t)$ 
	with intertimes $\mathcal{T}_j$ with the following distribution 
	\begin{equation}
	P\{\mathcal{T}>t\}\equiv P\{\mathcal{N}^{\alpha,\nu}(t)=0\}= E_{\nu,1}(-\lambda^{\alpha}t^{\nu}). 
	\end{equation}	
	Then we have that 
	\begin{align}
	\nonumber &P\{\mathcal{N}^{\alpha, \nu}(t)=k\}= P(\mathcal{T}_1+\dots+\mathcal{T}_k<t, 
	\mathcal{T}_1+\dots+\mathcal{T}_{k+1}>t)\\
	\nonumber &= P(	\mathcal{T}_1+\dots+\mathcal{T}_{k}<t)-	P(\mathcal{T}_1+\dots+\mathcal{T}_{k+1}<t)\\
     &=\int_0^{t}P\{\mathcal{T}_1+\dots+\mathcal{T}_{k}\in ds\}-
	\int_0^{t}P\{\mathcal{T}_1+\dots+\mathcal{T}_{k+1}\in ds\},\label{renew}
	\end{align}
	whose Laplace transform is given by
	\begin{align}
	&\int_0^{\infty}e^{-\gamma t}P\{\mathcal{N}^{\alpha,\nu}(t)= k\}dt\label{aa}\\
  \nonumber &=	\int_0^{\infty}e^{-\gamma t}\left(\int_0^{t}P\{\mathcal{T}_1+\dots+\mathcal{T}_{k}\in ds\}\right)dt-
	\int_0^{\infty}e^{-\gamma t}\left(	\int_0^{t}P\{\mathcal{T}_1+\dots+\mathcal{T}_{k+1}\in ds\}\right)dt\\
	\nonumber &= \frac{1}{\gamma}\int_0^{\infty} e^{-\gamma s}P\{\mathcal{T}_1+\dots+\mathcal{T}_k\in ds\}-\frac{1}{\gamma}\int_0^{\infty} e^{-\gamma s}P\{\mathcal{T}_1+\dots+\mathcal{T}_{k+1}\in ds\}\\
	\nonumber&= \frac{1}{\gamma}\bigg[\left(\int_0^{\infty}e^{-\gamma s}P\{\mathcal{T}_1\in ds\}\right)^k-\left(\int_0^{\infty}e^{-\gamma s}P\{\mathcal{T}_1\in ds\}\right)^{k+1}\bigg]= \frac{\lambda^{\alpha k}\gamma^{\nu-1}}{(\gamma^{\nu}+\lambda^{\alpha})^{k+1}},
	\end{align}
	where we have used the fact that
	\begin{equation}
	\int_0^{\infty}e^{-\gamma t}\lambda^{\alpha }t^{\nu-1}E_{\nu,\nu}(-\lambda^\alpha t^{\nu})dt= \frac{\lambda^{\alpha}}{\gamma^{\nu}+\lambda^{\alpha}}.
	\end{equation}
	We observe that the Laplace transform of the state probabilities of the renewal process previously considered, for $\alpha = 1$ coincides with the Laplace transform of the state probabilities of the time-fractional Poisson process studied by Beghin and Orsingher in \cite{be}. It is well-known that the time-fractional Poisson process is a renewal process with Mittag-Leffler distributed intertimes. The same is not true for the space-time fractional Poisson process. Indeed the Laplace transform \eqref{aa} does not coincide with the Laplace transform of the state probabilities of the space-time fractional Poisson processes for all $k\geq 0$, as can be shown by plain calculations.\\
	Moreover, we observe that we have the following interesting recursive relation for the Laplace transform of the space-time fractional Poisson process
	\begin{align}
	&\int_0^{\infty}e^{-\gamma t} P\{N^{\alpha,\nu}(t)=k\}dt= 
	\frac{(-1)^k}{k!}\sum_{m=0}^{\infty}\frac{(-\lambda)^{\alpha m}}{\gamma^{\nu m+1}}\frac{\Gamma(\alpha m +1)}{\Gamma(\alpha m+1-k)}\\
	\nonumber & = \frac{(-1)^k\lambda^k}{k!}\sum_{m=0}^{\infty}
	\frac{(-1)^{\alpha m}}{\gamma^{\nu m+1}}\frac{d^k}{d\lambda^k} \lambda^{\alpha m}
	= \frac{(-1)^k\lambda^k}{k!}\frac{d^k}{d\lambda^k} \frac{\gamma^{\nu-1}}{\lambda^{\alpha}+\gamma^{\nu}}\\
	\nonumber& =  \frac{(-1)^k\lambda^k}{k!}\frac{d^k}{d\lambda^k}\int_0^{\infty}e^{-\gamma t} P\{N^{\alpha, \nu}(t)= 0\}dt=  \frac{(-1)^k\lambda^k}{k!}\frac{d^k}{d\lambda^k}\int_0^{\infty}e^{-\gamma t} P\{\mathcal{N}^{\alpha, \nu}(t)= 0\} dt.
	\end{align}
	
	\section{Counting processes with Bern\u{s}tein intertimes: general results}
		    
		    In a recent paper (see \cite{bruno}), Orsingher and Toaldo have considered more general counting processes with Bern\u{s}tein intertimes and random jumps.
		    In this section we also show that the space-fractional Poisson process is a special case of this wide family of counting 
		    processes, that admits an explicit and simple form for
		    $P\{T_k<\infty\}$.\\
		    We briefly recall that the counting processes $N^f(t)$ considered in \cite{bruno} have independent and stationary increments, 
		    generalizing the classical Poisson process and their jumps have distribution given by
		    \begin{equation}\label{for1}
		    P\{N^f[t,t+dt)=k\}= \begin{cases}
		    dt\frac{\lambda^k}{k!}\int_0^{\infty}e^{-\lambda s}s^k \nu(ds)+o(dt), \quad &k\geq 1, \\
		    1-dt \int_0^{\infty}(1-e^{-\lambda s})\nu(ds)+o(dt), \quad &k=0,
		    \end{cases}
		    \end{equation} 
		    where 
		    \begin{equation}
		    f(\lambda)= \int_0^{\infty}(1-e^{-\lambda s})\nu(ds),
		    \end{equation}
		    is a Bern\u{s}tein function with L\'evy measure $\nu$.
		    In \cite{bruno} it was proved that the state probabilities $p_k^f(t)$ 
		    of the processes $N^f(t)$ satisfy the difference-differential equations
		    \begin{equation}
		    \frac{d}{dt}p^f_k(t)=-f(\lambda)p^f_k(t)+\sum_{m=1}^k\frac{\lambda^m}{m!}p^f_{k-m}(t)\int_0^\infty e^{-s\lambda}s^m \nu(ds), \quad k\geq 0, t>0,
		    \end{equation}
		    with the usual initial conditions and have the form 
			\begin{equation}\label{for2}
			P\{N^{f}(t)=k\}= p_k^f(t)= \frac{(-1)^k}{k!}\frac{d^k}{du^k}e^{-tf(\lambda u)}\bigg|_{u=1}.
			\end{equation}
	
			Moreover, the probability generating function of $N^f(t)$ is given by 
			\begin{equation}
			G^f(u,t)=e^{-tf(\lambda(1-u))}, \quad |u|\leq 1.
			\end{equation}
			Note that, for $\nu(ds) = \frac{\alpha s^{-\alpha-1}}{\Gamma(1-\alpha)}ds$, the space-fractional Poisson process is retrieved.
		
		By repeating the arguments of Theorem 2.1 and having in mind
		\eqref{for1} and \eqref{for2} we have that
		
		\begin{align}
	\nonumber	P\{T_k^f\in ds\}&=P\bigg\{ \bigcup_{j=0}^{k-1}\{N^f(s)= j, 
		N^f[s,s+ds) = k-j\}\bigg\}\\ &= \sum_{j=0}^{k-1}\frac{(-1)^j}{j!}\frac{d^j}{du^j}e^{-tf(\lambda u)}\bigg|_{u=1}\cdot dt \frac{\lambda^{k-j}}{(k-j)!}\int_0^{\infty}
								 	e^{-\lambda z}z^{k-j} \nu(dz), \label{uno}
		\end{align}
		for $s>0$.
		Therefore 
			\begin{align}\label{due}
			P\{T_k^f<\infty\}&= \sum_{j=0}^{k-1}\frac{(-1)^j}{j!}\frac{d^j}{du^j}\frac{1}{f(\lambda u)}\bigg|_{u=1}\cdot \frac{\lambda^{k-j}}{(k-j)!}\int_0^{\infty}
									 	e^{-\lambda z}z^{k-j} \nu(dz)\\
			\nonumber &=(-1)^{k-1}\sum_{j=0}^{k-1}\frac{1}{j!}\frac{d^j}{du^j}\frac{1}{f(\lambda u)}\bigg|_{u=1}\frac{\lambda^{k-j}}{(k-j)!}\frac{d^{k-j} }{d\lambda^{k-j}}f(\lambda)	.					 	
			\end{align}

	  Equation \eqref{uno} provides the hitting time distribution 
	  of level $k$ for the process $N^f(t)$ while 
	  \eqref{due} gives the hitting probability of level $k$.\\
	  The crucial point is the evaluation of $\frac{d^j}{du^j}\frac{1}{f(\lambda u)}\bigg|_{u=1}$ which seems possible for a small subset of the Bern\u{s}tein functions
	  of which $f(\lambda)= \lambda^\alpha$, $0<\alpha<1$, is a particular case.
	  \begin{os}
	  For all Bern\u{s}tein function $f$ we have that
	  \begin{equation}
	  P\{T_1^f<\infty\}<1.
	  \end{equation}
	  This is because 
	  \begin{align}
	  \nonumber P\{T_1^f<\infty\}&= \frac{\lambda}{f(\lambda)}\int_0^{\infty} e^{-\lambda s}s\nu(ds)= \frac{\displaystyle{\lambda \frac{d}{d\lambda}
	  \int_0^{\infty} (1-e^{-\lambda s)}\nu(ds)}}{\int_0^{\infty} (1-e^{-\lambda s})\nu(ds)}\\
	  \nonumber &= \lambda \frac{d}{d\lambda}\ln f(\lambda)<1, 
	  \end{align}
	  since $f(\lambda)<\lambda$ because $f(0)=0$ and $f''(\lambda)<0$.
	  \end{os}
	 	
	 \section{A generalization of counting processes with Bern\u{s}tein intertimes}
	
		In this section we consider the process related to the following Cauchy problem
		\begin{equation}\label{s-f2}
		 	\begin{cases}
		 	\frac{d}{dt}p_k(t)=-\displaystyle{\sum_{j=1}^n} f_j(\lambda(I-B))p_k(t)\\
		 	p_k(0)=\begin{cases}
		 	0 & k>0\\
		 	1 & k=0,
		 	\end{cases}
		 	\end{cases}
		 	\end{equation}
		 where $f_1, f_2,\dots, f_n$ are $n$ different Bern\u{s}tein functions. We will consider below point processes of the form
		 $ N\left(\sum_{j=1}^n H^{f_j}(t)\right)$ whose distribution is governed by equations \eqref{s-f2}. These processes are time-changed point processes where the role of time is played by the superposition of independent subordinators $H^{f_j}(t)$.
		 
		 \begin{te}
		 The distribution of the subordinated Poisson process
		 \begin{equation}\label{aioo}
		 \mathfrak{N}^{f_1,f_2, \dots, f_n}(t)= N\left(\sum_{j=1}^n H^{f_j}(t)\right), \quad t\geq 0
		 \end{equation}
		 is the solution to the Cauchy problem \eqref{s-f2}. In \eqref{aioo} $H^{f_1}(t), H^{f_2}(t), \dots,H^{f_n}(t)$ are $n$ independent subordinators with 
		 Bern\u{s}tein functions $f_1, f_2,\dots, f_n$, respectively.
		 \end{te}
		 
		 \begin{proof}
		 Let us determine the probability generating function of the r.v.
		 \eqref{aioo} 
		 \begin{align}
		 \mathbb{E}u^{\mathfrak{N}^{f_1,f_2, \dots, f_n}(t)}= 
		 \int_0^{\infty} \mathbb{E}u^{N(s)}P \bigg\{\sum_{j=1}^n H^{f_j}(t) \in ds  \bigg\}.
		 \end{align}
		 Then
		 \begin{align}
		 \nonumber\mathbb{E}u^{\mathfrak{N}^{f_1,f_2,\dots, f_n}(t)}&= \int_0^{\infty}
		 e^{-\lambda s(1-u)}P\bigg\{\sum_{j=1}^n H^{f_j}(t)\in ds\bigg\}\\
		 \nonumber & = \prod_{j=1}^n \int_0^{\infty}
		 		 e^{-\lambda s(1-u)}P\bigg\{H^{f_j}(t)\in ds\bigg\}\\
		 	 &= exp\left(-t\displaystyle{\sum_{j=1}^n} f_j(\lambda (1-u))\right).\label{pgr}
		 \end{align}
		 On the other hand, from \eqref{s-f2}, we can construct the equation governing the probability generating function $G$ by taking into account that (see \cite{bruno}, Remark 2.2 for details)
		 \begin{equation}
		 -f(\lambda(I-B))p_k(t)= -f(\lambda)p^f_k(t)+\sum_{m=1}^k\frac{\lambda^m}{m!}p^f_{k-m}(t)\int_0^\infty e^{-s\lambda}s^m \nu(ds).
		 \end{equation}
		 Thus from \eqref{s-f2}, by multiplying both members by $u^k$ and summing up w.r. to $k$, after calculation similar to those of Theorem 2.1 of \cite{bruno}, we arrive at the equation 
		 \begin{equation}\label{pgr2}
		 \begin{cases}
		 &\displaystyle{\frac{\partial G}{\partial t}}= -\displaystyle{\sum_{j=1}^n} f_j(\lambda (I-B))G(u,t)\\
		 &G(u,0)=1
		 \end{cases}
		 \end{equation}
		 It is now straightforward to check that \eqref{pgr} satisfies \eqref{pgr2}
		 \end{proof}
		 
		 \begin{os}
		 On the basis of the Theorem 4.1, we are also able to obtain explicitly the state probabilities of the subordinated process \eqref{aioo} for $n=2$, as follows
		 \begin{align}
		 \label{alse} p_k(t)&= \int_0^{\infty}\int_0^{\infty}e^{-\lambda s}
		 \frac{(\lambda(s+u))^k}{k!}e^{-\lambda u}P\{H^{f_1}(t)\in ds\}
		 P\{H^{f_2}(t)\in du\}\\
		 \nonumber & = \frac{\lambda^k}{k!}\sum_{j=0}^{k}
		 \binom{k}{j} \int_0^{\infty}\int_0^{\infty}e^{-\lambda s}s^j u^{k-j}
		 e^{-\lambda u}P\{H^{f_1}(t)\in ds\}
		 	 P\{H^{f_2}(t)\in du\}\\
		 \nonumber &= \frac{\lambda^k}{k!}\sum_{j=0}^{k}
		 	 \binom{k}{j}\left[\int_0^{\infty}e^{-\lambda s}s^j 
		 	 P\{H^{f_1}(t)\in ds\}\right]\left[\int_0^{\infty}
		 	 	 e^{-\lambda u}  u^{k-j} \ P\{H^{f_2}(t)\in du\} \right]\\
		 \nonumber &= (-1)^k\frac{\lambda^k}{k!} \sum_{j=0}^{k}
		 	 	 \binom{k}{j}\frac{d^j}{d\lambda^j}e^{-t f_1(\lambda)}
		 	 	 \frac{d^{k-j}}{d\lambda^{k-j}}e^{-tf_2(\lambda)}	 .
		 \end{align}
		 It is simple to prove that, in the case where $f_1(s)= s^{\alpha}$, with $\alpha\in (0,1)$ and $f_2=0$ 
		 we recover the explicit form of the state probabilites of the space-fractional Poisson process (see Theorem 2.2 in \cite{fede}).
		 For $f_2 = 0$, formula \eqref{alse} yields $p_k(t)$ for the process $N^{f}(t)$. It is clear that only in the case $n = 2$, the state probabilities can be calculated in explicit form in a rather simple way. 
		  \end{os}
		 
		 \smallskip
		 
		 We now consider the following subordinated Poisson process
		 \begin{equation}\label{inve}
		 \mathfrak{N}^{f_1,f_2, \dots, f_n}(L^{\nu}(t))= N\left(\sum_{j=1}^n H^{f_j}(L^{\nu}(t))\right),
		 \end{equation}
		 where $L^{\nu}(t)$ is the inverse of the stable subordinator $H^{\nu}(t)$ and $H^{f_1}(t), H^{f_2}(t), \dots,H^{f_n}(t)$ are $n$ independent subordinators with 
		 Bern\u{s}tein functions $f_1, f_2,\dots, f_n$, respectively.
		 The distribution of $L^\nu(t)$ is obtained from that of $H^\nu(t)$
		 because 
		 \begin{equation}
		 P\{L^{\nu}(t)>s\}= P\{H^{\nu}(s)<t\}
		 \end{equation} 
		 and thus
		 \begin{equation}
		 P\{L^{\nu}(t)\in ds\}= -\frac{d}{ds}\int_0^t P\{H^{\nu}(s)\in dz\}.
		 \end{equation}
		 We have the following result.
		 \begin{te}
		 The distribution of the subordinated Poisson process \eqref{inve}
		 		 is the solution to the Cauchy problem
		 \begin{equation}\label{s-f3}
		   \begin{cases}
		 		 	\displaystyle{\frac{d^\nu}{dt^\nu}}p_k(t)=-\displaystyle{\sum_{j=1}^n} f_j(\lambda(I-B))p_k(t), \quad \nu\in(0,1)\\
		 		 	p_k(0)=\begin{cases}
		 		 	0 & k>0\\
		 		 	1 & k=0,
			\end{cases}
		 	\end{cases}
		 \end{equation}
		  involving time-fractional derivatives in the sense of Dzerbayshan--Caputo.		 
		 \end{te}
		 
		 \begin{proof}
		 
		 The p.g.f. of \eqref{inve} is given by
		 \begin{align}
		 \nonumber &\mathbb{E}u^{\mathfrak{N}^{f_1,f_2, \dots, f_n}(L^{\nu}(t))}= 
		 \int_0^{\infty} \mathbb{E} u^{\mathfrak{N}^{f_1,f_2, \dots, f_n}(s)}P\{L^{\nu}(t)\in ds \}\\
		  &  =-\int_0^{\infty}e^{-s\sum_{j=1}^n f_j(\lambda(1-u))}
		 \bigg\{\frac{d}{ds}\int_0^t P\{H^{\nu}(s)\in dz\}\bigg\} ds.\label{pgfa}
		 \end{align}
		 The Laplace transform of \eqref{pgfa} becomes
		 \begin{align}
		 \nonumber &\int_0^{\infty} e^{-\gamma t} \mathbb{E}u^{\mathfrak{N}^{f_1,f_2, \dots,f_n}(L^{\nu}(t))} dt\\
		 \nonumber & = -\int_0^{\infty} dt \ e^{-\gamma t} \int_0^{\infty}e^{-s\sum_{j=1}^n f_j(\lambda (1-u))}
		 	 \bigg\{\frac{d}{ds}\int_0^t P\{H^{\nu}(s)\in dz\}\bigg\}ds\\
		 \nonumber &=- \int_0^{\infty} e^{-s\sum_{j=1}^n f_j(\lambda (1-u))}\frac{d}{ds}\bigg\{
		 \int_0^{\infty} P\{H^{\nu}(s)\in dz\} \int_z^{\infty}
		 e^{-\gamma t}dt\bigg\}ds\\
		 \nonumber &= -\int_0^{\infty}e^{-s\sum_{j=1}^n f_j(\lambda (1-u))} \frac{1}{\gamma}\frac{d}{ds}\bigg\{\int_0^{\infty}e^{-\gamma z}P\{H^{\nu}(s)\in dz\} \bigg\}ds\\
		 \nonumber &=\gamma^{\nu-1}\int_0^{\infty}e^{-\gamma^{\nu}s-s\sum_{j=1}^n f_j(\lambda(1-u))} ds\\
		 &= \frac{\gamma^{\nu-1}}{\gamma^{\nu}+\displaystyle{\sum_{j=1}^n} f_j(\lambda(1-u))}. \label{bell}
		 \end{align}
		 The inverse Laplace transform of \eqref{bell} yields the p.g.f
		 of the process $\mathfrak{N}^{f_1,f_2, \dots f_n}(L^{\nu}(t))$ as 
		 \begin{equation}\label{bell1}
		 \mathbb{E}u^{\mathfrak{N}^{f_1,f_2, \dots,f_n}(L^{\nu}(t))} = 
		 E_{\nu, 1}\left(-t^{\nu}\bigg[\sum_{j=1}^n f_j(\lambda(1-u))\bigg]\right).
		 \end{equation} 
		 On the other hand, by multiplying both terms of \eqref{s-f3} for 
		 $u^k$ and summing over all $k$ we have 
		 \begin{equation}
		 		 \begin{cases}
		 		 &\displaystyle{\frac{\partial^\nu G}{\partial t^\nu}}= -\displaystyle{\sum_{j=1}^n} f_j(\lambda (1-u))G(u,t)\\
		 		 &G(u,0)=1,
		 		 \end{cases}
		 		 \end{equation}
		whose solution clearly coincides with \eqref{bell1} as claimed. 
		 \end{proof}

		 From the p.g.f \eqref{bell1} we can extract the mean value of $\mathfrak{N}^{f_1,f_2, \dots, f_n}(L^{\nu}(t))$ as 
		 \begin{align}
		 \nonumber &\mathbb{E}\mathfrak{N}^{f_1,f_2, \dots, f_n}(L^{\nu}(t))= 
		 \frac{d}{du}E_{\nu, 1}\left(-t^{\nu}\bigg[\sum_{j=1}^n f_j(\lambda (1-u))\bigg]\right)\bigg|_{u=1}\\
		 \nonumber & =\frac{\lambda t^{\nu}}{\nu}E_{\nu, \nu}\left(-t^{\nu}\bigg[\sum_{j=1}^n f_j(0)\bigg]\right)\sum_{j=1}^{n} f_j'(0)\\
		\nonumber &= \frac{\lambda t^{\nu}}{\Gamma(\nu+1)}\sum_{j=1}^n  f_j'(0)=
		\frac{\lambda t^{\nu}}{\Gamma(\nu+1)}\sum_{j=1}^n \ \int_0^{\infty}
		 s \ \nu_j(ds),
		 \end{align}
		 which for $n=2$, $f_1= x$ and $f_2 = 0$, yields the mean-value of the time-fractional Poisson process (see \cite{be}).
		 In the special case of L\'evy measure equal to 
		 \begin{equation}
		 \nu_j(ds)= \frac{\alpha_j s^{-\alpha_j-1}e^{-\theta_j s}}{\Gamma(1-\alpha_j)}ds,
		 \end{equation}
		 with Bern\u{s}tein function $f_j(x)= (x+\theta_j)^{\alpha_j}-\theta_j^{\alpha_j}$, $\alpha_j \in (0,1)$, we have that
		 \begin{equation}
		 \mathbb{E}\mathfrak{N}^{f_1,f_2}(L^{\nu}(t))= \frac{ \lambda t^{\nu}}{\Gamma(\nu+1)} \sum_{j=1}^n\left(\alpha_j \theta_j^{\alpha_j-1}\right).
		 \end{equation}
		 We also notice that
		 \begin{align}
		\nonumber &\mbox{Var} \ \mathfrak{N}^{f_1,f_2, \dots,f_n}(L^\nu(t))=\lambda
		t\bigg[\lambda\sum_{j=1}^n f_j''(0)+\sum_{j=1}^n f_j'(0)\bigg]
		 \end{align}
	
	\section{Poisson process with Generalized Grey Brownian clocks}
	
	In a previous paper \cite{be}, Beghin and Orsingher have proved that the time-fractional Poisson process $N^{\nu}(t)$ can be represented as
	\begin{equation}
	 N^{\nu}(t)\stackrel{d}{=}N(\mathcal{T}_{2\nu}(t)), \quad t>0,
	\end{equation}
	where $N(\cdot)$ is the homogeneous Poisson process with rate $\lambda$ and the random time variable $\mathcal{T}_{2\nu}(t)$ possesses density 
	obtained by folding the solution of the time-fractional diffusion equation. A relevant consequence of this result is given by the special case $\nu = 1/2$, where the process $N^{1/2}(t)$ becomes
	\begin{equation}
		 N^{1/2}(t)\stackrel{d}{=}N(|B(t)|), \quad t>0,
	\end{equation}
	where $B(t)$ is the standard Brownian motion.
	Starting from this analysis, further results about Poisson processes with different Brownian clocks have then been obtained in \cite{stoca}.\\
	Following these ideas, we here consider the Poisson process at generalized grey Brownian times, namely $N(|\mathcal{G}_{H,\nu}(t)|)$.
	The generalized grey Brownian motion (ggBm) $\mathcal{G}_{H,\nu}(t)$ is a non-Markovian stochastic process recently introduced in the literature by Mura and coauthors (see for example \cite{mura,mura1})
	to model anomalous diffusions. 
		The ggBm includes as special cases the Brownian motion (for $\nu = 2H = 1 $), the fractional Brownian motion (for $\nu = 1$) and
		time-fractional diffusions (for $2H = 1$ and $0<\nu<1$). We also observe, that in the recent paper \cite{gianni2}, the authors 
		gave a physical motivation for the application of the ggBm related to random walks in random complex media.  		
	The ggBm $\mathcal{G}_{H,\nu}(t)$ has probability law satisfying
	the Cauchy problem
	\begin{equation}\label{ggbm}
	\begin{cases}
	&{}^C\left(t^{1-2H}\frac{\partial}{\partial t}\right)^{\nu} u(y,t)=
	c^2 \frac{\partial^2}{\partial y^2}u(y,t), \quad \nu \in (0,1), H\in (0,1), y\in \mathbb{R}, t>0,\\
	& u(y,0)= \delta(y).
	\end{cases}
	\end{equation}
	The fractional operator appearing in equation \eqref{ggbm} can be expressed in terms of Erd\'elyi-Kober integral operators, according to the McBride theory \cite{Mcbride} (we refer to the Appendix for more details on this point).\\
	It is well-known that the Fourier transform of the solution of \eqref{ggbm} is given by (see \cite{jmp})
	\begin{equation}
	\int_{-\infty}^{+\infty} e^{i\gamma y}u(y,t)dy = 
	E_{\nu,1}\left(-\frac{c^2}{(2H)^\nu}\gamma^2t^{2H\nu}\right),
	\end{equation} 
	whose inverse is given by
	\begin{equation}
	u(y,t)= \frac{(2H)^{\nu/2}}{2 c t^{\nu H}} W_{-\nu/2,1-\nu/2}
	\left(-\frac{|y|(2H)^{\nu/2}}{c t^{\nu H}}\right),
	\end{equation}
	where $W_{-\nu/2,1-\nu/2}(\cdot)$ is the well-known Wright function.
	In order to simplify our calculations we assume that 
	$$c^2 = (2H)^\nu.$$
	
	We have now the following result
	\begin{te}
	For Poisson process at a generalized grey Brownian time the following equality in distribution holds
	\begin{equation}
	 N(|\mathcal{G}_{H,\nu}(t)|)\stackrel{d}{=}\mathcal{N}^{H,\nu}(t),
	\end{equation}
	where $\mathcal{N}^{H,\nu}(t)$, $t>0$, is the fractional Poisson process whose state probabilities $p_k(t)$ are governed by the following fractional difference-differential equations
	\begin{equation}\label{genn}
		\begin{cases}
		\displaystyle{{}^C\left(t^{1-2H}\frac{\partial}{\partial t}\right)^{\nu/2}}p_k(t)=
		-\lambda'(1-B) p_k(t), \quad \nu\in(0,2], \ H\in(0,1)\\
		p_k(0)=\begin{cases}
		0 & k>0\\
		1 & k=0
		\end{cases}
		\end{cases}
		\end{equation}
    where $\lambda'= \lambda (2H)^{\nu/2}$.
	\end{te}
	\begin{proof}
	The p.g.f. of the process $N(|\mathcal{G}_{H,\nu}(t)|)$ is given by
	\begin{equation}
	G(u,t)=\int_0^{\infty} e^{\lambda y(u-1)}\bar{u}(y,t)dy, 
	\end{equation}
	where $\bar{u}(y,t)$ is obtained by folding the fundamental solution of \eqref{ggbm} (for $c^2 = (2H)^\nu$) and is given by (see \cite{jmp}, Theorem 3.1)
	\begin{equation}
	\bar{u}(y,t)= \begin{cases} \displaystyle{\frac{1}{t^{H\nu}}W_{-\nu/2, 1-\nu/2}\left(-\frac{y}{t^{H\nu}}\right)},\quad & y\geq 0, t>0,\\
	0, \quad & y<0,
	\end{cases}
	\end{equation}
	with $\nu>-1$.
	Therefore we have that 
	\begin{equation}\label{gen 1}
	G(u,t)=\int_0^{\infty} e^{\lambda y(u-1)}\frac{1}{t^{H\nu}}W_{-\nu/2, 1-\nu/2}\left(-\frac{y}{t^{H\nu}}\right)dy.
	\end{equation}
	In the next steps, we will use the following integral representations of the Wright function and of the Mittag-Leffler function on the Hankel path $Ha$ (see e.g. \cite{be}, p.1799 and \cite{gorenflo})
	\begin{align}
	W_{-\nu/2, 1-\nu/2}(x)= \frac{1}{2\pi i}\int_{Ha}\frac{e^{z+x z^{\nu/2}}}{z^{1-\nu/2}}dz, \label{wri}\\
	E_{\nu/2,1}(x)= \frac{1}{2\pi i }\int_{Ha}\frac{e^z z^{\nu/2-1}}{z^{\nu/2}-x}dz \label{mitta}.
	\end{align}
	By using \eqref{wri} in \eqref{gen 1}, we obtain
	\begin{align}
	\nonumber G(u,t)&= \frac{1}{2\pi i {t^{H\nu}}}\int_0^{\infty}
	e^{\lambda y(u-1)}dy\int_{Ha}\frac{e^{z-yz^{\nu/2}t^{-H\nu}}}{z^{1-\nu/2}}dz\\
	\nonumber & =\frac{1}{2\pi i {t^{H\nu}}}\int_{Ha}\frac{e^z}{z^{1-\nu/2}}dz
	\int_0^{\infty}e^{\lambda y(u-1)-yz^{\nu/2}t^{-H\nu}}dy \\
	\nonumber &= \frac{1}{2\pi i {t^{H\nu}}}\int_{Ha}\frac{e^z z^{\nu/2-1}}{z^{\nu/2}t^{-\nu H}-\lambda (u-1)}dz\\
	&= E_{\nu/2,1}\left(-\lambda (1-u)t^{H\nu}\right),\label{pgre}
	\end{align}
	where in the last step, we used equation \eqref{mitta}.\\
	On the other hand, the p.g.f. of the process $\mathcal{N}^{H,\nu}(t)$
	can be simply obtained by multiplying both terms of \eqref{genn} by $u^k$ and summing over all $k$, we obtain
	\begin{equation}
	\begin{cases}
	{}^C\left(t^{1-2H}\frac{\partial}{\partial t}\right)^{\nu/2} G(u,t)= -\lambda' (1-u) G(u,t),\\
	G(u,0)= 1,
	\end{cases}
	\end{equation}
	whose solution (see \cite{jmp} for detailed calculations) is given by
	\begin{equation}\label{pgff}
	G(u,t)=  E_{\nu/2,1}\left(-\frac{\lambda'}{(2H)^{\nu/2}} (1-u)t^{ H\nu}\right),
	\end{equation}
	that coincides with \eqref{pgre} iff $\lambda'= \lambda (2H)^{\nu/2}$
	as claimed.
	\end{proof}	
	
	\begin{coro}
	The Poisson process composed with a fractional Brownian motion time $N(|B_H(t)|)$
	can be represented as 
		\begin{equation}\label{pdif}
		 N(|B_H(t)|)\stackrel{d}{=}\mathcal{N}^{H,2}(t),
		\end{equation}
		where $\mathcal{N}^{H,2}(t)$, $t>0$, is the counting process whose state probabilities $p_k(t)$ are governed by the following difference-differential equations
		\begin{equation}
			\begin{cases}
			{}^C\left(t^{1-2H}\frac{\partial}{\partial t}\right)p_k(t)=
			-\lambda'(1-B) p_k(t)\\
			p_k(0)=\begin{cases}
			0 & k>0\\
			1 & k=0
			\end{cases}
			\end{cases}
			\end{equation}
	\end{coro}
	\begin{os}
	We observe that for $H= 1/2$ and $\nu = 1$ in \eqref{pdif} we recover the Poisson process at a reflected Brownian motion whose state probabilities are governed by the classical fractional difference-differential equations of order $1/2$ as proved by Beghin and Orsingher in \cite{be}.
	\end{os}
	
	It is also simple to prove that this kind of counting process with random time is not a renewal process, as we are
	going to show by means of explicit calculations. From the p.g.f. \eqref{pgff} we obtain the explicit form of the state probabilities of 
	the process $\mathcal{N}^{H,\nu}(t)$ 
	\begin{equation}
	p_k(t)= \frac{(\lambda t^{H\nu})^k}{k!}\sum_{m=0}^{\infty}\frac{(m+k)!}{m!}\frac{(-\lambda t^{H\nu})^m}{\Gamma(\frac{\nu}{2}(k+m)+1)},
	\end{equation}
	with $\lambda = (2H)^{\nu/2}\lambda$, whose Laplace transform is given by
	\begin{equation}\label{latra}
	\int_0^\infty e^{-st} p_k(t)dt= \frac{\lambda^k}{k!}\sum_{m=0}^{\infty} \frac{(m+k)!}{m!}\frac{\Gamma(H\nu(m+k)+1)}{\Gamma(\frac{\nu}{2}(m+k)+1)}
	\frac{(-\lambda)^m}{s^{H\nu(m+k)+1}}.
	\end{equation}
	Let us now consider the renewal process $\tilde{N}(t)$ with the following distribution of the i.i.d. intertimes 
		\begin{equation}
		P\{\mathfrak{T}>t\}\equiv P\{\mathcal{N}^{H,\nu}(t)=0\}= E_{\nu/2,1}(-\lambda t^{H\nu}). 
		\end{equation}	
		Then we have that (see equation \eqref{renew}) 
		\begin{align}
		\nonumber &P\{\tilde{N}(t)=k\}=\int_0^{t}P\{\mathfrak{T}_1+\dots+\mathfrak{T}_{k}\in ds\}-
		\int_0^{t}P\{\mathfrak{T}_1+\dots+\mathfrak{T}_{k+1}\in ds\},
		\end{align}
		whose Laplace transform is given by
		\begin{align}
		&\int_0^{\infty}e^{-\gamma t}P\{\tilde{N}(t)= k\}dt\label{latra1}\\
	  \nonumber &=	\int_0^{\infty}e^{-\gamma t}\left(\int_0^{t}P\{\mathfrak{T}_1+\dots+\mathfrak{T}_{k}\in ds\}\right)dt-
		\int_0^{\infty}e^{-\gamma t}\left(	\int_0^{t}P\{\mathfrak{T}_1+\dots+\mathfrak{T}_{k+1}\in ds\}\right)dt\\
		\nonumber &= \frac{1}{\gamma}\int_0^{\infty} e^{-\gamma s}P\{\mathfrak{T}_1+\dots+\mathfrak{T}_k\in ds\}-\frac{1}{\gamma}\int_0^{\infty} e^{-\gamma s}P\{\mathfrak{T}_1+\dots+\mathfrak{T}_{k+1}\in ds\}\\
		\nonumber&= \frac{1}{\gamma}\bigg[\left(\int_0^{\infty}e^{-\gamma s}P\{\mathfrak{T}_1\in ds\}\right)^k-\left(\int_0^{\infty}e^{-\gamma s}P\{\mathfrak{T}_1\in ds\}\right)^{k+1}\bigg]\\
		\nonumber &=
		 \frac{1}{\gamma}\bigg[K(\gamma)\bigg]^k\bigg\{1-K(\gamma)\bigg\} ,
		\end{align}
		where we have used the fact that
		\begin{equation}\nonumber
		K(\gamma)=2H\lambda\int_0^{\infty}e^{-\gamma t}t^{H\nu-1}E_{\nu/2,\nu/2}(-\lambda t^{H\nu})dt = 2H\lambda \sum_{m=0}^\infty \frac{(-\lambda)^m}{\Gamma(\frac{\nu}{2}(m+1))}\frac{\Gamma(H\nu+H\nu m)}{\gamma^{H\nu+H\nu m}}.
		\end{equation}
	    Clearly equations \eqref{latra} and \eqref{latra1} do not coincide and therefore we can conclude that $\mathcal{N}^{H,\nu}(t)$
		is not a renewal process.\\
		We consider in the next Proposition the case of a counting process with Bern\u{s}tein intertimes randomized by means of
		the random time $\mathcal{G}_{H,\nu}(t)$.
		\begin{prop}
		The process whose state probabilities are governed by the fractional difference-differential equations
				\begin{equation}
					\begin{cases}
					{}^C\left(t^{1-2H}\frac{\partial}{\partial t}\right)^{\nu/2}p_k(t)=
					-f\left(\lambda(1-B)\right) p_k(t)\\
					p_k(0)=\begin{cases}
					0 & k>0\\
					1 & k=0
					\end{cases}
					\end{cases}
					\end{equation}
		admits the following representation
		\begin{equation}
		N(H^f(|\mathcal{G}_{H,\nu}(t)|)),
		\end{equation}
		where $H^f$ is the subordinator with Bern\u{s}tein function
		$f$.
		\end{prop}
		The proof directly follows by applying the results presented in \cite{bruno} and also by taking 
		into account the previous analysis.

	 \section{Appendix: Fractional operators with time-varying coefficients}
	 In this Appendix, we briefly recall some useful preliminaries about the theory of fractional powers of operators with time-varying coefficients. 
	 In \cite{Mcbride} McBride considered the generalized hyper-Bessel operators
	         \begin{equation}
	             \label{L}
	             L=t^{a_1}\frac{\mathrm d}{\mathrm{d}t}t^{a_2}\dots t^{a_n}\frac{\mathrm d}{\mathrm{d}t}t^{a_{n+1}},\qquad t>0,
	         \end{equation}
	         where $n$ is an integer number and $a_1,\dots, a_{n+1}$ are real numbers. The operator $L$ defined in \eqref{L} acts on the functional space
	                 \begin{equation}
	                     F_{p,\mu}=\{f \colon t^{-\mu}f(t)\in F_p\},
	                 \end{equation}
	                 where
	                 \begin{equation}
	                     F_p=\left\{f\in C^{\infty} \colon t^k \frac{\mathrm d^k f}{\mathrm dt^k} \in L^p(0, \infty), \: k=0,1,\dots\right\},
	                 \end{equation}
	                 for $1 \leq p < \infty$ and for any real number $\mu$.
	          
	         Let us introduce the following constants related to the general operator $L$.
	         \begin{align*}
	 			a=\sum_{k=1}^{n+1}a_k, \qquad m= |a-n|,
	 			\qquad b_k= \frac{1}{m}\left(\sum_{i=k+1}^{n+1}a_i+k-n\right), \quad k=1, \dots, n.
	         \end{align*}
	         The definition of the fractional hyper-Bessel-type operator is given by
	   	    \begin{definition}
	         Let $m= n-a>0$, $f\in F_{p,\mu}$ and
	         \begin{align*}
	             b_k\in A_{p,\mu,m}:=\left\{\eta \in \mathbb{R}
	             \colon m\eta+\mu+m\neq \frac{1}{p}-ml, \: l=0, 1, 2,\dots\right\}, \qquad k=1,\dots, n.
	         \end{align*}
	         Then
	         \begin{equation}
	          	\label{pot}
	             L^{\alpha}f=m^{n\alpha}t^{-m\alpha}\prod_{k=1}^{n}I_{m}^{b_k,-\alpha}f,
	         \end{equation}
	         where, for $\alpha >0$ and $m\eta+\mu+m > \frac{1}{p}$
	         \begin{equation}
	             \label{mc1-2}
	             I_m^{\eta, \alpha}f=
	             \frac{t^{-m\eta-m\alpha}}{\Gamma(\alpha)}\int_0^t(t^m-u^m)^{\alpha-1}u^{m\eta}f(u)\, \mathrm{d}(u^m),
	         \end{equation}
	         and for $\alpha\leq 0$
	         \begin{equation}
	             \label{mc2}
	               	I_m^{\eta, \alpha}f=(\eta+\alpha+1)I_m^{\eta, \alpha+1}f+\frac{1}{m} I_m^{\eta, \alpha+1}
	               	\left(t\frac{\mathrm{d}}{\mathrm{d}t}f\right).
	   	        \end{equation}
	   	    \end{definition}
	         For a full discussion about this approach to fractional operators we refer to \cite{Mcbride}.
	         
	         The regularized Caputo-like counterpart of the operator \eqref{pot} was introduced in \cite{jmp} in order to obtain meaningful results for the fractional diffusion with time-varying coefficients.
	         \begin{definition}
	         	Let $\alpha$ be a positive real number, $m= n-a>0$, $f\in F_{p,\mu}$ is such that
	             \begin{equation*}
	             	L^\alpha\left( f(t) - \displaystyle\sum_{k=0}^{b-1}\frac{t^{k}}{k!} f^{(k)}(0^+)\right)
	             \end{equation*}
	             exists. Then we define ${}^C L^\alpha$ by
	             \begin{equation}
	 	            \label{pot1}
	 	            {}^C  L^\alpha f(t) =  L^\alpha\left( f(t) - \sum_{k=0}^{b-1}\frac{t^{k}}{k!} f^{(k)}(0^+)\right),
	             \end{equation}
	             where $b = \lceil \alpha \rceil$.
	         \end{definition}
	 Concluding this Appendix, the operator ${}^C\left(t^{1-2H}\frac{\partial}{\partial t}\right)^{2\nu}$ appearing in equation \eqref{ggbm} is obtained by specializing the coefficients in the general definitions given above.

    \end{document}